\documentclass[leqno,12pt]{amsart}
\usepackage{amssymb}
\usepackage{amsmath}
\usepackage{enumerate}
\usepackage{amsfonts}
\usepackage{hyperref}
\usepackage{soulutf8}
\usepackage{tikz}

\hypersetup{
    colorlinks=true,
    linkcolor= blue,
    citecolor =cyan,
    urlcolor = teal,
}

\newtheorem{theorem}{Theorem}[section]
\newtheorem{corollary}[theorem]{Corollary}
\newtheorem{lemma}[theorem]{Lemma}
\newtheorem{proposition}[theorem]{Proposition}

\numberwithin{equation}{section}

\newcommand{\supp}{\text{\rm supp}\,}
\newcommand{\diam}{\text{\rm diam}}

\begin{document}
\title[Singular integrals in the rational Dunkl setting ]{Singular integrals in the rational Dunkl setting }
 
\author[ J. Dziuba\'nski and A. Hejna]{Jacek Dziuba\'nski and Agnieszka Hejna}


\begin{abstract}
On $\mathbb R^N$ equipped with a normalized root system $R$ and a multiplicity function $k\geq 0$ let us consider a (non-radial) kernel $K(\mathbf x)$ which has properties similar to those from the classical theory. We prove that a singular integral Dunkl convolution operator associated with the kernel $K$ is bounded on $L^p$ for $1<p<\infty$ and of weak-type (1,1). Further we study a maximal function related to the Dunkl convolutions with truncation of $K$.
\end{abstract}

\address{J. Dziuba\'nski and A. Hejna, Uniwersytet Wroc\l awski,
Instytut Matematyczny,
Pl. Grunwaldzki 2/4,
50-384 Wroc\l aw,
Poland}
\email{jdziuban@math.uni.wroc.pl}
\email{hejna@math.uni.wroc.pl}

\thanks{
Research supported by the National Science Centre, Poland (Narodowe Centrum Nauki), Grant 2017/25/B/ST1/00599.}

\maketitle

\section{Introduction}

\par 
The aim of this note is to study singular integral convolution operators  in the Dunkl setting.  We fix a  normalized  root system $R$ in $\mathbb R^N$  and a multiplicity function $k\geq 0$. Let $dw(\mathbf x)$ denote the associated measure and $\mathbf N$ the homogeneous dimension (see Section \ref{sec:preliminaries}).  For a positive integer $s$ consider a kernel   $K\in C^{s} (\mathbb R^N\setminus \{0\})$ such that 
\begin{equation}\label{eq:uni_on_annulus}\tag{A} \sup_{0<a<b<\infty} \Big| \int_{a<\|\mathbf x\|<b} K(\mathbf x)\, dw(\mathbf x)\Big|<\infty,  
\end{equation} 
\begin{equation}\label{eq:assumption1}\tag{D} 
\Big|\frac{\partial^\beta}{\partial \mathbf x^\beta} K(\mathbf x)\Big|\leq C\|\mathbf x\|^{-\mathbf N-|\beta|} \quad \text{for all} \ |\beta |\leq s. 
\end{equation}
Set 
$$K^{\{t\}}(\mathbf x)=K(\mathbf x)(1-\phi(t^{-1} \mathbf x)),$$ where $\phi$ is a  fixed  radial $C^\infty$-function    supported by the unit ball $B(0,1)$ such that $\phi (\mathbf x)=1$ for $\|\mathbf x\|<1/2$. 
We  prove that if $s$ is sufficiently large, 
 then there are constants $C_p>0$ independent of $t>0$ such that 
 \begin{align*}
     \| f*K^{\{t\}} \|_{L^p(dw)}\leq C_p\| f\|_{L^p(dw)} \quad \text{for} \ 1<p<\infty
 \end{align*}
and 
\begin{align*}
    w(\{\mathbf x\in\mathbb R^N:| f*K^{\{t\}}(\mathbf x)|>\lambda\} )\leq C_1\lambda^{-1}\| f\|_{L^1(dw)}
\end{align*}
(Theorems~\ref{theorem:weak_type_truncated1} and~\ref{theorem:strong_type_truncated1}), where the symbol $*$ denotes the Dunkl convolution. 
We also prove that under the additional assumption 
\begin{equation}\label{eq:limitA}\tag{L} 
     \lim_{\varepsilon \to 0} \int_{\varepsilon <|\mathbf x|<1} K(\mathbf x)\, dw(\mathbf x)=L,
     \end{equation} 
where $L\in \mathbb C$,  the limit $\lim_{t\to 0} f*K^{\{t\}} (\mathbf x)$ exists and defines a bounded operator on $L^p(dw)$ for $1<p<\infty$, which is of weak-type (1,1) as well (Theorem~\ref{theorem:weak_type_K}, see also Theorem~\ref{theorem:main_L2}). Moreover, in this case, 
the maximal operator 
$$ K^*f(\mathbf x)=\sup_{t>0} |f*K^{\{t\}}(\mathbf x)|$$
is bounded on $L^p(dw)$ for $1<p<\infty$ and of weak-type (1,1) (Theorem~\ref{theorem:maximal}).

If $k\equiv 0$, then $dw$ is the Lebesgue measure in $\mathbb R^N$ and the Dunkl convolution reduces to  the classical one. So the the above results are well known and $s=1$ suffices in this case (see i.e.~\cite[Chapter 5]{Duo},~\cite{St1},~\cite{St2}). However, in the general case  of $R$ and $k$ 
the main difficulty which one faces trying to study singular integral operators in the Dunkl setting  lies  in the lack of knowledge about boundendess of the so called Dunkl translations $\tau_{\mathbf x}$ on $L^p(dw)$-spaces for $p\ne 2$.  Consequently, it is not known if for a fixed non-radial $L^1$-function $f$  the Dunkl convolution operator $g\mapsto f*g$ is bounded on $L^p(dw)$.  The recent observations made in~\cite{DzH} allow us to obtain some knowledge for the functions $\tau_{\mathbf y}f(\mathbf x)$ provided  $f$ satisfies certain regularity conditions in smoothness and decay. In the present paper we explore and extend these ideas of~\cite{DzH}  for proving  boundedness of singular integral convolution operators provided $s=s_{0}$ in \eqref{eq:assumption1}, where $s_0$ is the smallest even integer bigger than $\mathbf{N}/2$. 

\section{Preliminaries and notation}\label{sec:preliminaries}

The Dunkl theory is a generalization of the Euclidean Fourier analysis. It started with the seminal article \cite{Dunkl} and developed extensively afterwards (see e.g. \cite{RoeslerDeJeu}, \cite{Dunkl0}, \cite{Dunkl3}, \cite{Dunkl2}, \cite{GR}, \cite{Roesler2}, \cite{Roesle99}, \cite{Roesler2003}, \cite{ThangaveluXu}, \cite{Trimeche2002}). 
In this section we present basic facts concerning the theory of the Dunkl operators.  For details we refer the reader to~\cite{Dunkl},~\cite{Roesler3}, and~\cite{Roesler-Voit}. 

We consider the Euclidean space $\mathbb R^N$ with the scalar product $\langle\mathbf x,\mathbf y\rangle=\sum_{j=1}^N x_jy_j
$, $\mathbf x=(x_1,...,x_N)$, $\mathbf y=(y_1,...,y_N)$, and the norm $\| \mathbf x\|^2=\langle \mathbf x,\mathbf x\rangle$. For a nonzero vector $\alpha\in\mathbb R^N$,  the reflection $\sigma_\alpha$ with respect to the hyperplane $\alpha^\perp$ orthogonal to $\alpha$ is given by
\begin{align*}
\sigma_\alpha (\mathbf x)=\mathbf x-2\frac{\langle \mathbf x,\alpha\rangle}{\| \alpha\| ^2}\alpha.
\end{align*}
In this paper we fix a normalized root system in $\mathbb R^N$, that is, a finite set  $R\subset \mathbb R^N\setminus\{0\}$ such that   $\sigma_\alpha (R)=R$ and $\|\alpha\|=\sqrt{2}$ for every $\alpha\in R$. The finite group $G$ generated by the reflections $\sigma_\alpha \in R$ is called the {\it Weyl group} ({\it reflection group}) of the root system. A~{\textit{multiplicity function}} is a $G$-invariant function $k:R\to\mathbb C$ which will be fixed and $\geq 0$  throughout this paper. 
 Let
\begin{align*}
dw(\mathbf x)=\prod_{\alpha\in R}|\langle \mathbf x,\alpha\rangle|^{k(\alpha)}\, d\mathbf x
\end{align*} 
be  the associated measure in $\mathbb R^N$, where, here and subsequently, $d\mathbf x$ stands for the Lebesgue measure in $\mathbb R^N$.
We denote by $\mathbf N=N+\sum_{\alpha \in R} k(\alpha)$ the homogeneous dimension of the system. Clearly, 
\begin{align*} w(B(t\mathbf x, tr))=t^{\mathbf N}w(B(\mathbf x,r)) \ \ \text{\rm for all } \mathbf x\in\mathbb R^N, \ t,r>0 
\end{align*}
and
\begin{align*}
\int_{\mathbb R^N} f(\mathbf x)\, dw(\mathbf x)=\int_{\mathbb R^N} t^{-\mathbf N} f(\mathbf x\slash t)\, dw(\mathbf x)\ \ \text{for} \ f\in L^1(dw)  \   \text{\rm and} \  t>0.
\end{align*}
Observe that (\footnote{The symbol $\sim$ between two positive expressions means that their ratio remains between two positive constants.})
\begin{align*} w(B(\mathbf x,r))\sim r^{N}\prod_{\alpha \in R} (|\langle \mathbf x,\alpha\rangle |+r)^{k(\alpha)},
\end{align*}
so $dw(\mathbf x)$ is doubling, that is, there is a constant $C>0$ such that
\begin{equation}\label{eq:doubling} w(B(\mathbf x,2r))\leq C w(B(\mathbf x,r)) \ \ \text{ for all } \mathbf x\in\mathbb R^N, \ r>0.
\end{equation}

For $\xi \in \mathbb{R}^N$, the {\it Dunkl operators} $T_\xi$  are the following $k$-deformations of the directional derivatives $\partial_\xi$ by a  difference operator:
\begin{align*}
     T_\xi f(\mathbf x)= \partial_\xi f(\mathbf x) + \sum_{\alpha\in R} \frac{k(\alpha)}{2}\langle\alpha ,\xi\rangle\frac{f(\mathbf x)-f(\sigma_\alpha \mathbf x)}{\langle \alpha,\mathbf x\rangle}.
\end{align*}
The Dunkl operators $T_{\xi}$, which were introduced in~\cite{Dunkl}, commute and are skew-symmetric with respect to the $G$-invariant measure $dw$.
 Suppose that $\xi\ne 0$, $f,g \in C^1(\mathbb{R}^N)$ and $g$ is radial. The following Leibniz rule can be confirmed by a direct calculation:
\begin{align*}
T_{\xi}(f g)=f(T_\xi g)+g(T_\xi f). 
\end{align*}
For fixed $\mathbf y\in\mathbb R^N$ the {\it Dunkl kernel} $E(\mathbf x,\mathbf y)$ is the unique analytic solution to the system
\begin{equation}\label{eq:Dunkl_kernel_definition}
    T_\xi f=\langle \xi,\mathbf y\rangle f, \ \ f(0)=1.
\end{equation}
The function $E(\mathbf x ,\mathbf y)$, which generalizes the exponential  function $e^{\langle \mathbf x,\mathbf y\rangle}$, has the unique extension to a holomorphic function on $\mathbb C^N\times \mathbb C^N$. Moreover, it satisfies $E(\mathbf{x},\mathbf{y})=E(\mathbf{y},\mathbf{x})$ for all $\mathbf{x},\mathbf{y} \in \mathbb{C}^N$.

Let $\{e_j\}_{1 \leq j \leq N}$ denote the canonical orthonormal basis in $\mathbb R^N$ and let $T_j=T_{e_j}$. For multi-index $\beta=(\beta_1,\beta_2,\ldots,\beta_N)  \in\mathbb N_0^N$, we set 
$$ |\beta|=\beta_1+\beta_2 +\ldots +\beta_N,$$
$$\partial^{\beta}=\partial_1^{\beta_1} \circ \partial_2^{\beta_2}\circ \ldots \circ \partial_N^{\beta_N},$$
$$ T^\beta = T_1^{\beta_1}\circ T_2^{\beta_2}\circ \ldots \circ T_N^{\beta_N}.$$
In our further consideration we shall need the following lemma.
\begin{lemma}
For all $\mathbf{x} \in \mathbb{R}^N$, $\mathbf{z} \in \mathbb{C}^N$ and $\nu \in \mathbb{N}_0^{N}$ we have
$$|\partial^{\nu}_{\mathbf{z}}E(\mathbf{x},\mathbf{z})| \leq \|\mathbf{x}\|^{|\nu|}\exp(\|\mathbf{x}\|\|{\rm Re \;}\mathbf{z}\|).$$
In particular, 
\begin{align*} | E(i\xi, \mathbf x)|\leq 1 \quad \text{ for all } \xi,\mathbf x\in \mathbb R^N.
\end{align*}
\end{lemma}
\begin{proof}
See~\cite[Corollary 5.3]{Roesle99}.
\end{proof}

\begin{corollary}\label{coro:Roesler}
There is a constant $C>0$ such that for all $\mathbf{x},\xi \in \mathbb{R}^N$ we have
\begin{equation}
|E(i\xi,\mathbf{x})-1| \leq C\|\mathbf{x}\|\|\xi\|.
\end{equation}
\end{corollary}

The \textit{Dunkl transform}
  \begin{align*}\mathcal F f(\xi)=c_k^{-1}\int_{\mathbb R^N} E(-i\xi, \mathbf x)f(\mathbf x)\, dw(\mathbf x),
  \end{align*}
  where
  $$c_k=\int_{\mathbb{R}^N}e^{-\frac{\|\mathbf{x}\|^2}{2}}\,dw(\mathbf{x})>0,$$
   originally defined for $f\in L^1(dw)$, is an isometry on $L^2(dw)$, i.e.,
   \begin{equation}\label{eq:Plancherel}
       \|f\|_{L^2(dw)}=\|\mathcal{F}f\|_{L^2(dw)} \text{ for all }f \in L^2(dw),
   \end{equation}
and preserves the Schwartz class of functions $\mathcal S(\mathbb R^N)$ (see \cite{deJeu}). Its inverse $\mathcal F^{-1}$ has the form
  \begin{align*} \mathcal F^{-1} g(x)=c_k^{-1}\int_{\mathbb R^N} E(i\xi, \mathbf x)g(\xi)\, dw(\xi).
  \end{align*}
The {\it Dunkl translation\/} $\tau_{\mathbf{x}}f$ of a function $f\in\mathcal{S}(\mathbb{R}^N)$ by $\mathbf{x}\in\mathbb{R}^N$ is defined by
\begin{align*}
\tau_{\mathbf{x}} f(\mathbf{y})=c_k^{-1} \int_{\mathbb{R}^N}{E}(i\xi,\mathbf{x})\,{E}(i\xi,\mathbf{y})\,\mathcal{F}f(\xi)\,{dw}(\xi).
\end{align*}
  It is a contraction on $L^2(dw)$, however it is an open  problem  if the Dunkl translations are bounded operators on $L^p(dw)$ for $p\ne 2$.
  
  The following specific formula was obtained by R\"osler \cite{Roesler2003} for the Dunkl translations of (reasonable) radial functions $f({\mathbf{x}})=\tilde{f}({\|\mathbf{x}\|})$:
\begin{equation}\label{eq:translation-radial}
\tau_{\mathbf{x}}f(-\mathbf{y})=\int_{\mathbb{R}^N}{(\tilde{f}\circ A)}(\mathbf{x},\mathbf{y},\eta)\,d\mu_{\mathbf{x}}(\eta)\text{ for all }\mathbf{x},\mathbf{y}\in\mathbb{R}^N.
\end{equation}
Here
\begin{equation*}
A(\mathbf{x},\mathbf{y},\eta)=\sqrt{{\|}\mathbf{x}{\|}^2+{\|}\mathbf{y}{\|}^2-2\langle \mathbf{y},\eta\rangle}=\sqrt{{\|}\mathbf{x}{\|}^2-{\|}\eta{\|}^2+{\|}\mathbf{y}-\eta{\|}^2}
\end{equation*}
and $\mu_{\mathbf{x}}$ is a probability measure, 
which is supported in the set $\operatorname{conv}\mathcal{O}(\mathbf{x})$,  where $\mathcal O(\mathbf x) =\{\sigma(\mathbf x): \sigma \in G\}$ is the orbit of $\mathbf x$. Formula~\eqref{eq:translation-radial} implies that for all radial $f \in L^1(dw)$ and $\mathbf{x} \in \mathbb{R}^N$ we have
\begin{align*}
    \|\tau_{\mathbf{x}}f(\mathbf{y})\|_{L^1(dw(\mathbf{y}))} \leq \|f(\mathbf{y})\|_{L^1(dw(\mathbf{y}))}.
\end{align*}

  {The \textit{Dunkl convolution\/} $f*g$ of two reasonable functions (for instance Schwartz functions) is defined by
$$
(f*g)(\mathbf{x})=c_k\,\mathcal{F}^{-1}[(\mathcal{F}f)(\mathcal{F}g)](\mathbf{x})=\int_{\mathbb{R}^N}(\mathcal{F}f)(\xi)\,(\mathcal{F}g)(\xi)\,E(\mathbf{x},i\xi)\,dw(\xi) \text{ for }\mathbf{x}\in\mathbb{R}^N,
$$
or, equivalently, by}
\begin{align*}
  {(f{*}g)(\mathbf{x})=\int_{\mathbb{R}^N}f(\mathbf{y})\,\tau_{\mathbf{x}}g(-\mathbf{y})\,{dw}(\mathbf{y})=\int_{\mathbb R^N} f(\mathbf y)g(\mathbf x,\mathbf y) \,dw(\mathbf{y}) \text{ for all } \mathbf{x}\in\mathbb{R}^N},  
\end{align*}
where, here and subsequently, $g(\mathbf x,\mathbf y)=\tau_{\mathbf x}g(-\mathbf y)$. 

By an interpolation argument, if $1\leq p\leq 2$ and $q=2p/(2-p)$, then
\begin{align*}
  \|f*g\|_{L^q(dw)}\leq \| f\|_{L^2(dw)}\| g\|_{L^p(dw)}.  
\end{align*}

The {\it Dunkl Laplacian} associated with $R$ and $k$  is the differential-difference operator $\Delta=\sum_{j=1}^N T_{j}^2$, which  acts on $C^2(\mathbb{R}^N)$-functions by

\begin{align*}
    \Delta f(\mathbf x)=\Delta_{\rm eucl} f(\mathbf x)+\sum_{\alpha\in R} k(\alpha) \delta_\alpha f(\mathbf x),
\end{align*}
\begin{align*}
    \delta_\alpha f(\mathbf x)=\frac{\partial_\alpha f(\mathbf x)}{\langle \alpha , \mathbf x\rangle} - \frac{\|\alpha\|^2}{2} \frac{f(\mathbf x)-f(\sigma_\alpha \mathbf x)}{\langle \alpha, \mathbf x\rangle^2}.
\end{align*}
Obviously, $\mathcal F(\Delta f)(\xi)=-\| \xi\|^2\mathcal Ff(\xi)$. The operator $\Delta$ is essentially self-adjoint on $L^2(dw)$ (see for instance \cite[Theorem\;3.1]{AH}) and generates the semigroup $e^{t\Delta}$  of linear self-adjoint contractions on $L^2(dw)$. The semigroup has the form
  \begin{align*}
  e^{t\Delta} f(\mathbf x)=\mathcal F^{-1}(e^{-t\|\xi\|^2}\mathcal Ff(\xi))(\mathbf x)=\int_{\mathbb R^N} h_t(\mathbf x,\mathbf y)f(\mathbf y)\, dw(\mathbf y),
  \end{align*}
  where the heat kernel 
  \begin{equation}\label{eq:heat_def}
      h_t(\mathbf x,\mathbf y)=\tau_{\mathbf x}h_t(-\mathbf y), \ \ h_t(\mathbf x)=\mathcal F^{-1} (e^{-t\|\xi\|^2})(\mathbf x)=c_k^{-1} (2t)^{-\mathbf N\slash 2}e^{-\| \mathbf x\|^2\slash (4t)}
  \end{equation}
  is a $C^\infty$-function of all variables $\mathbf x,\mathbf y \in \mathbb{R}^N$, $t>0$, and satisfies \begin{align*} 0<h_t(\mathbf x,\mathbf y)=h_t(\mathbf y,\mathbf x),
  \end{align*}
 \begin{align*} \int_{\mathbb R^N} h_t(\mathbf x,\mathbf y)\, dw(\mathbf y)=1.
 \end{align*}
  Set
$$V(\mathbf x,\mathbf y,t)=\max (w(B(\mathbf x,t)),w(B(\mathbf y, t))).$$
Let 
$$d(\mathbf x,\mathbf y)=\min_{\sigma\in G}\| \sigma(\mathbf x)-\mathbf y\|$$
be the distance of the orbit of $\mathbf x$ to the orbit of $\mathbf y$.
The following theorem was proved in~\cite[Theorem 3.1]{DzH1} (see also~\cite[Theorem 4.1]{ADzH}).
\begin{theorem}\label{theorem:heat}
There are constants $C,c>0$ such that for all $\mathbf{x},\mathbf{y} \in \mathbb{R}^N$ and $t>0$ we have
\begin{equation}\label{eq:Gauss}
h_t(\mathbf{x},\mathbf{y}) \leq C\Big(1+\frac{\|\mathbf{x}-\mathbf{y}\|}{t}\Big)^{-2}\,V(\mathbf{x},\mathbf{y},\!\sqrt{t\,})^{-1}\,e^{-\hspace{.25mm}c\hspace{.5mm}d(\mathbf{x},\mathbf{y})^2\slash t}.
\end{equation}
Moreover, if $\|\mathbf{y}-\mathbf{y}'\| \leq \sqrt{t}$, then
\begin{equation}\label{eq:Gauss-Lipschitz}
|h_t(\mathbf{x},\mathbf{y})-h_t(\mathbf{x},\mathbf{y}')| \leq C\frac{\|\mathbf{y}-\mathbf{y}'\|}{\sqrt{t}}\Big(1+\frac{\|\mathbf{x}-\mathbf{y}\|}{t}\Big)^{-2}\,V(\mathbf{x},\mathbf{y},\!\sqrt{t\,})^{-1}\,e^{-\hspace{.25mm}c\hspace{.5mm}d(\mathbf{x},\mathbf{y})^2\slash t}.
\end{equation}

\end{theorem}

Theorem~\ref{theorem:heat} and~\eqref{eq:translation-radial} imply the following Lemma (see~\cite[Corollary 3.5]{DzH1}).

\begin{lemma}\label{lem:nonradial_estimation}
Suppose that $\varphi \in C_c^{\infty}(\mathbb{R}^N)$ is radial and supported by the unit ball. Then there is $C>0$ such that for all $\mathbf{x},\mathbf{y}\in \mathbb{R}^{N}$ and $t>0$ we have
\begin{align*}
    &|\varphi_t(\mathbf{x},\mathbf{y})| \leq C \Big(1+\frac{\|\mathbf{x}-\mathbf{y}\|}{t}\Big)^{-2}V(\mathbf{x},\mathbf{y},t)^{-1}\chi_{[0,1]}(d(\mathbf{x},\mathbf{y})/t).
\end{align*}
\end{lemma}

\section{ \texorpdfstring{$L^2(dw)$}{L2} estimates}
In the present section we assume that $K(\mathbf x)$ satisfies \eqref{eq:uni_on_annulus} and \eqref{eq:assumption1} with $s=1$, that is, 
\begin{equation}\label{eq:K_der1} \tag{$\text{D}(1)$}
    |\partial^{\beta}K(\mathbf{x})| \leq C_{\beta}\|\mathbf{x}\|^{-\mathbf{N}-|\beta|} \text{ for }|\beta| \leq 1.
\end{equation}

Our aim is to we prove that convolution operators with the truncated kernels $K^{\{t\}}$ are uniformly bounded on $L^2(dw).$ Then we add the assumption  \eqref{eq:limitA} and show the  $L^2$-bound of the limiting operator.

We start by the easy observation that \eqref{eq:K_der1} implies 
\begin{equation}\label{eq:K_der} \tag{$\text{D'}(1)$}
    |T^{\beta}K(\mathbf{x})| \leq C_{\beta}\|\mathbf{x}\|^{-\mathbf{N}-|\beta|} \text{ for }|\beta| \leq 1.
\end{equation}

Let $\phi $ be a fixed $C^\infty(\mathbb{R}^N)$ radial function supported in the unit ball such  that  $\phi(\mathbf x)=1$ for all $\mathbf{x} \in \mathbb{R}^N$ such that $ \| \mathbf x\|\leq 1/2$.

{

For $0<a<b<\infty$, let 
$$ K_{a,\infty}(\mathbf x)=K(\mathbf x)\chi_{\{\mathbf y: a<\| \mathbf y\|\}}(\mathbf x) \quad \text{and} \quad  K_{a,b}(\mathbf x)=K(\mathbf x)\chi_{\{\mathbf y: a<\| \mathbf y\|<b\}}(\mathbf x),$$ 
$$ K^{\{a\}}(\mathbf x)=K(\mathbf x)(1-\phi(a^{-1}\mathbf x)) \quad \text{and}\quad K^{\{a,b\}}(\mathbf x)=K^{\{a\}}(\mathbf x)-K^{\{b\}}(\mathbf x).$$
Let us list the following easily proved properties of the truncated kernels which follow from \eqref{eq:K_der} and~\eqref{eq:uni_on_annulus}:  
$$K_{a,b}, \, K^{\{a,b\}}\in L^1(dw)\cap L^2(dw), \quad  K_{a,\infty}, \, K^{\{a\}}\in L^2(dw),$$ 
$$\text{supp}\, K^{\{a,b\}}\subseteq \{\mathbf x\in\mathbb R^N: a/2\leq \| \mathbf x\|\leq b\},$$ 
\begin{equation}\label{eq:L1compare} \sup_{0<a<b<\infty} \| K_{a,b}-K^{\{a,b\}}\|_{L^1(dw)}=C_0<\infty,  
\end{equation}
$$\lim_{b\to \infty} \| K_{a,b}-K_{a,\infty}\|_{L^2(dw)}=0, \quad \lim_{b\to \infty} \| K^{\{a,b\}}-K^{\{a\}}\|_{L^2(dw)}=0.$$
Consequently, by the Plancherel  identity (see~\eqref{eq:Plancherel}),
\begin{equation}\label{eq:L^2limit}
\lim_{b \to \infty} \| \mathcal FK_{a,b}-\mathcal F K_{a,\infty}\|_{L^2(dw)}=0, \quad \lim_{b \to \infty} \| \mathcal FK^{\{a,b\}}-\mathcal F K^{\{a\}}\|_{L^2(dw)}=0.
\end{equation}
Moreover, there are constants $A',C>0$ such that for all  $0<a<b<\infty$  one has
\begin{align*}
    |K^{\{a,b\}}(\mathbf x)|\leq C\| \mathbf x\|^{-\mathbf N}, \quad  |K^{\{a\}}(\mathbf x)|\leq C\| \mathbf x\|^{-\mathbf N},
\end{align*}
 \begin{align*}    |T_jK^{\{a,b\}}(\mathbf x)| \leq C\| \mathbf x\|^{-\mathbf N-1}, \quad  |T_jK^{\{a\}}(\mathbf x)| \leq C\| \mathbf x\|^{-\mathbf N-1},
\end{align*}
\begin{align}\label{newA}    \Big|\int K^{\{a,b\}}(\mathbf x)\, dw(\mathbf x)\Big| \leq A'.
\end{align}
\begin{lemma}\label{lem:uniform_boundedness0}
\eqref{eq:K_der} and~\eqref{eq:uni_on_annulus} imply that there is a constant $C>0$ such that for all $0<a<b<\infty$ one has  $|\mathcal{F}K_{a,b}(\xi)| \leq C$ and $|\mathcal{F}K^{\{a,b\}}(\xi)| \leq C$ .

\end{lemma}
\begin{proof}
Thanks to \eqref{eq:L1compare} it suffices to prove the second inequality. Assume first that  $\xi \in \mathbb{R}^N$ satisfies $a \leq \|\xi\|^{-1} \leq b$. Put $t=\|\xi\|^{-1}$. We have $K^{\{a,b\}}=K^{\{a,t\}}+K^{\{t,b\}}$ and, consequently, 
\begin{equation}\label{eq:first_splitting}
\begin{split}
    \mathcal{F}K^{\{a,b\}}(\xi)=\mathcal{F}K^{\{a,t\}}(\xi)+\mathcal{F}K^{\{t,b\}}(\xi)
    &=:I_1+I_2.
\end{split}
\end{equation}
In order to estimate $I_1$, we write
\begin{align*}
    I_1=c_k^{-1}\int K^{\{a,t\}} (\mathbf{x})(E(-i\xi,\mathbf{x})-1)\,dw(\mathbf{x})+c_k^{-1}\int K^{\{a,t\}}(\mathbf{x})\,dw(\mathbf{x})=:I_{1,1}+I_{1,2}.
\end{align*}
Clearly, by~\eqref{newA} we get $|I_{1,2}| \leq C$. For $I_{1,1}$, by Corollary~\ref{coro:Roesler} and~\eqref{eq:K_der}, we obtain
\begin{align*}
    |I_{1,1}| \leq C\|\xi\|\int_{\|\mathbf{x}\| \leq t}\|\mathbf{x}\|^{-\mathbf{N}+1}\,dw(\mathbf{x}) \leq C.
\end{align*}
We now turn to estimate $I_2$.  Choose $j \in \{1,2,\ldots,N\}$ such that $|\xi_j| \geq {N}^{-1/2}\|\xi\|$. Then, thanks to~\eqref{eq:Dunkl_kernel_definition}, we have
\begin{align*}
        |I_2|& \leq C\sqrt{N}\|\xi\|^{-1} \Big|\int_{\mathbb{R}^N} K^{\{t,b\}}(\mathbf x)T_jE(-i
        \xi ,\mathbf x )\, dw(\mathbf x)\Big|\\
        &\leq C\sqrt{N}\|\xi\|^{-1} \Big|\int_{\mathbb{R}^N} T_j K^{\{t,b\}}(\mathbf x)E(-i
        \xi ,\mathbf x )\, dw(\mathbf x)\Big|\\
        &\leq C\sqrt{N}\|\xi\|^{-1} \Big|\int_{\frac{1}{2} t\leq\|\mathbf x\| } \| \mathbf x\|^{-\mathbf N-1} \, dw(\mathbf x)\Big| \leq C'.
\end{align*}
The cases $\|\xi\|^{-1}< a $ or $\|\xi\|^{-1}>b$ can be treated similarly (we have to deal with just one integral in~\eqref{eq:first_splitting}).
\end{proof}

From \eqref{eq:L^2limit} and Lemma \ref{lem:uniform_boundedness0} we easily deduce the following corollary.  
\begin{corollary}\label{coro:uniform_K_a} 
\eqref{eq:K_der} and~\eqref{eq:uni_on_annulus} imply that there is a constant $C>0$ such that for every $0<a<\infty$ we have  $$\|\mathcal FK^{\{a\}}\|_{L^\infty} \leq C, \quad \text{and} \quad  \|\mathcal FK_{a,\infty}\|_{L^\infty} \leq C.$$
\end{corollary}
 
\begin{corollary}\label{uniform_L2}
There is a constant $C>0$ such that for all $0<a<b<\infty$ we have 
\begin{equation*}
    \| K^{\{a,b\}}*f\|_{L^2(dw)}+  \| K_{a,b}*f\|_{L^2(dw)}+ 
   \| K^{\{a\}}*f\|_{L^2(dw)}+ \| K_{a,\infty}*f\|_{L^2(dw)}  \leq C\| f\|_{L^2(dw)}.
\end{equation*}
Moreover, for every $a>0$  and $f\in L^2(dw)$ we have
\begin{equation}\label{eq:L2strong}
\lim_{b\to\infty}  \|K^{\{a,b\}} f-K^{\{a\}}f   \|_{L^2(dw)}=0. 
\end{equation}
\end{corollary}
\begin{proof}
The  corollary follows directly  from Lemma \ref{lem:uniform_boundedness0},  Corollary \ref{coro:uniform_K_a}, the Plancherel identity~\eqref{eq:Plancherel}, and the definition of the Dunkl convolution.  
\end{proof}
 From now on to the end of the paper  we assume additionally  that \eqref{eq:limitA} is satisfied by $K$ as well. 

\begin{lemma}\label{lem:L}
Property~\eqref{eq:limitA} implies 
 \begin{align*}
     \lim_{a\to 0} \int_{\| \mathbf x\|<1} K^{\{a\}} (\mathbf x)\, dw(\mathbf x)=L.
 \end{align*}
 \end{lemma}
 \begin{proof}
 Let us define $\widetilde{\phi}(\|\mathbf{x}\|)=\phi(\mathbf{x})$. For $0< r\leq 1$ we set
 \begin{equation}\label{eq:A_F}
     F(r)=\int_{r<\|\mathbf{x}\|<1}K(\mathbf{x})\,dw(\mathbf{x})=\int_r^{1}\int_{\mathbb{S}^{N-1}}r_1^{N-1}K(r_1\omega)\prod_{\alpha \in R}|\langle r_1\omega,\alpha \rangle|^{k(\alpha)}\,d{\boldsymbol\sigma}(\omega)\,dr_1,
 \end{equation}
 where $\boldsymbol\sigma$ is the spherical measure. Note that $F$ is differentiable and satisfies
 \begin{align*}
     F'(r)=-\int_{\mathbb{S}^{N-1}}r^{N-1}K(r\omega)\prod_{\alpha \in R}|\langle r\omega,\alpha \rangle|^{k(\alpha)}\,d{\boldsymbol\sigma}(\omega).
 \end{align*}
 Consequently, by the definition of $K^{\{a\}}$ and integration by parts we have
 \begin{align*}
     \int_{\| \mathbf x\|<1} K^{\{a\}} (\mathbf x)\, dw(\mathbf x)&=\int_0^1 \big(1- \widetilde{\phi}(a^{-1}r)\big) \int_{\mathbb{S}^{N-1}}r^{N-1}K(r\omega)\prod_{\alpha \in R}|\langle r\omega,\alpha \rangle|^{k(\alpha)}\,d{\boldsymbol\sigma}(\omega)\,dr\\&=\int_0^1 \big(1- \widetilde{\phi}(a^{-1}r)\big)(-F'(r))\,dr\\
     &=-\int_0^1 (\widetilde{\phi}(a^{-1}r))'F(r)\,dr.
 \end{align*}
 We write
 \begin{equation}\label{eq:A_split}
 \begin{split}
    - \int_0^1 (\widetilde{\phi}(a^{-1}r))'F(r)\,dr&=-\int_0^1 (\widetilde{\phi}(a^{-1}r))'(F(r)-L)\,dr-L\int_0^1 (\widetilde{\phi}(a^{-1}r))'\,dr\\&=:S_1(a)+S_2(a).
 \end{split}
 \end{equation}
 Clearly,  for all $a<1/4$, we have
 \begin{equation}\label{eq:lim_A_2}
     S_2(a)=L.
 \end{equation}
 Note that $\text{supp}\, (\widetilde{\phi}(a^{-1}\cdot))'\subseteq [a/2,a]$, hence  
 \begin{align*}
     |S_1(a)| &\leq \int_0^1| (\widetilde{\phi}(a^{-1}r))'||F(r)-L| \,dr
     =\int_{a/2}^{a}|(\widetilde{\phi}(a^{-1}r))'||F(r)-L| \,dr \\&\leq \max_{r \in [a/2,a]}|F(r)-L|\int_{a/2}^{a}|a^{-1}\widetilde{\phi}'(a^{-1}r)|\,dr \leq C\|(\widetilde{\phi})'\|_{L^{\infty}}\max_{r \in [0,a]}|F(r)-L|.
 \end{align*}
Consequently, thanks to~\eqref{eq:limitA} and~\eqref{eq:A_F} we obtain $\lim_{a \to 0}S_1(a)=0$. Combining this fact with~\eqref{eq:lim_A_2} and~\eqref{eq:A_split} we get the claim.
 \end{proof}
\begin{lemma}\label{lem:K_a}
Under~\eqref{eq:K_der}, \eqref{eq:uni_on_annulus}, and \eqref{eq:limitA}  for almost every $\xi\in\mathbb R^N$, the limit
\begin{align*}
    \lim_{a\to 0} \mathcal FK_{a,\infty}(\xi)    
\end{align*}
exists and defines a bounded function denoted by  $\mathcal FK(\xi)$.
\end{lemma}
\begin{proof}
 According to Corollary ~\ref{coro:uniform_K_a}, it suffices to show that $\mathcal FK_{a,\infty}(\xi)$ is a Cauchy sequence as $a\to 0$. To this end we write
 \begin{align*}
|\mathcal FK_{a_1,\infty}(\xi)-\mathcal FK_{a_2,\infty}(\xi)|
&\leq c_k^{-1}\Big| \int_{a_1\leq \|\mathbf x\| \leq a_2} K(\mathbf x) \Big(E(-i\xi, \mathbf x) -1\Big)\, dw(\xi) \Big|\\
&\quad + c_k^{-1}\Big|\int_{a_1\leq \|\mathbf x\| \leq a_2} K(\mathbf x)\, dw(\mathbf x)\Big|=:I_1+I_2.
 \end{align*}
 Thanks to Corollary \ref{coro:Roesler},
 $$ I_1\leq  C\|\xi\|  \int_{a_1\leq \|\mathbf x\| \leq a_2} \|\mathbf x\|^{-\mathbf N +1}  dw(\mathbf x)\to 0 $$
 as $a_1$ and $a_2$ tend to $0^+$. The convergence of $I_2$ to 0 is a consequence of \eqref{eq:limitA}.
\end{proof}
As a consequence of Corollary~\ref{coro:uniform_K_a} and Lemma \ref{lem:K_a} we obtain the following theorem.
\begin{theorem}\label{theorem:limiting_operator_L2}
Under~\eqref{eq:K_der}, \eqref{eq:uni_on_annulus}, and \eqref{eq:limitA}, for every $f\in L^2(dw)$ the  limit 
\begin{align*}
    \lim_{a\to 0} K_{a,\infty}*f
\end{align*}
exists in $L^2(dw)$ and defines a operator, which is bounded on $L^2(dw)$ and will be denoted by $Kf$. Moreover, 
\begin{align*}
    Kf=\mathcal F^{-1} (\mathcal FK(\cdot)\mathcal Ff(\cdot)).
\end{align*}
\end{theorem}

\begin{theorem}\label{theorem:main_L2}
Under~\eqref{eq:K_der}, \eqref{eq:uni_on_annulus}, and \eqref{eq:limitA}  for almost every $\xi\in\mathbb R^N$,
\begin{align*}
  \lim_{ a\to 0} \mathcal FK^{\{a\}}(\xi)  = \mathcal FK(\xi),
\end{align*}
where $\mathcal FK(\xi)$ is defined in Lemma~\ref{lem:K_a}. Moreover, for every $f\in L^2(dw)$ the limit 
\begin{align*}
    \lim_{a\to 0} K^{\{a\}}*f    
\end{align*}
exists in $L^2(dw)$ and is equal to $Kf$ (see Theorem \ref{theorem:limiting_operator_L2}). 
\end{theorem}
\begin{proof}
Let $J(a,\xi)=\mathcal F( K^{\{a,1\}})(\xi)-\mathcal F( K_{a,1})(\xi)$. It suffices to show that $\lim_{a\to 0^+}J(a,\xi)=0$ for all $\xi\in\mathbb R^N$. Similarly to the proof of Lemma~\ref{lem:K_a}, we write
\begin{align*}
J(a,\xi)=&c_k^{-1}\int_{\mathbb{R}^N} \Big(K^{\{a,1\}}(\mathbf x)-K_{a,1}(\mathbf x)\Big)\Big( E(-i\xi,\mathbf x)-1\Big)\, dw(\mathbf x)\\
&+ c_k^{-1}\int_{\mathbb{R}^N} \Big(K^{\{a,1\}}(\mathbf x)-K_{a,1}(\mathbf x)\Big)\, dw(\mathbf x)=:J_1(a,\xi)+J_2(a,\xi).
\end{align*}
By Lemma~\ref{lem:L} we have $\lim_{a\to 0^+}J_2(a,\xi)=0$.  To deal with $J_1(a,\xi)$ we note that $$\text{supp}\,(K^{\{a,1\}}-K_{a,1})\subseteq \{\mathbf x: a/2\leq \| \mathbf x\|\leq a\} \quad \text{and}\quad  |K^{\{a,1\}}(\mathbf x)-K_{a,1}(\mathbf x)|\leq C\| \mathbf x\|^{-\mathbf N}.$$
Hence, using Corollary \ref{coro:Roesler}, we get 
$$J_1\leq C\int_{a/2\leq \| \mathbf x\|\leq a} \|\xi\| \|\mathbf x\|^{-\mathbf N+1}\, dw(\mathbf x)\to 0\quad \text{as} \quad a\to 0. $$
\end{proof}

\section{ \texorpdfstring{$L^p(dw)$}{Lp} estimates}

The purpose of this section is to study singular integrals operators on $L^p(dw)$. For this purpose we need to make the following stronger assumption on the kernel $K$, namely that~\eqref{eq:assumption1} holds for $|\beta|\leq s_0$, where $s_0$ is the smallest even positive integer bigger than $\mathbf N\slash 2$, that is,
   \begin{equation}\label{eq:K_der2}\tag{$\text{D}(s_0)$}
    |\partial^{\beta}K(\mathbf{x})| \leq C_{\beta}\|\mathbf{x}\|^{-\mathbf{N}-|\beta|} \text{ for all }|\beta| \leq s_0.
\end{equation}
Clearly,~\eqref{eq:K_der2} implies
\begin{equation}\label{eq:K_der2_Dunkl}
    |T^{\beta}K(\mathbf{x})| \leq C_{\beta}\|\mathbf{x}\|^{-\mathbf{N}-|\beta|} \text{ for all }|\beta| \leq s_0.
\end{equation}
{The goal of this section is to prove the following theorems.
\begin{theorem}\label{theorem:weak_type_truncated1}
Under~\eqref{eq:K_der2} and \eqref{eq:uni_on_annulus}, there is a constant $C>0$ such that for all  $0<a < b<\infty $,  $\lambda>0$, and $f \in L^1(dw)\cap L^2(dw)$ we have
\begin{equation}\label{eq:Psi_nm_weak}
    w(\{\mathbf{x} \in \mathbb{R}^N\,:\,|K^{\{a,b\}}f(\mathbf{x})|>\lambda\}) \leq C\lambda^{-1}\|f\|_{L^1(dw)},
\end{equation}
\begin{equation}\label{eq:Psi_n_weak}
     w(\{\mathbf{x} \in \mathbb{R}^N\,:\,|K^{\{a\}}f(\mathbf{x})|>\lambda\}) \leq C\lambda^{-1}\|f\|_{L^1(dw)}.
\end{equation}
\end{theorem}

\begin{theorem}\label{theorem:strong_type_truncated1}
Let $1<p<\infty$. Under~\eqref{eq:K_der2} and \eqref{eq:uni_on_annulus}, there is a constant $C=C_p>0$ such that for all $0<a < b<\infty$, and $f \in L^p(dw)\cap L^2(dw)$ we have
\begin{equation}\label{eq:Psi_nm_strong}
    \|K^{\{a,b\}} f\|_{L^p(dw)} \leq C\|f\|_{L^p(dw)},
\end{equation}
\begin{equation}\label{eq:Psi_n_strong}
    \|K^{\{a\}} f\|_{L^p(dw)} \leq C\|f\|_{L^p(dw)}.
\end{equation}
Moreover, the operators $K^{\{a,b\}}$ converge strongly to $K^{\{a\}}$ in $L^p(dw)$ as $b \to \infty$.
\end{theorem}
As the consequence of Theorems~\ref{theorem:weak_type_truncated1} and~\ref{theorem:strong_type_truncated1} we obtain the following theorem for the operator $K$ defined in Theorem~\ref{theorem:limiting_operator_L2}.

\begin{theorem}\label{theorem:weak_type_K}
Let $1 \leq p <\infty$. Under~\eqref{eq:K_der2},  \eqref{eq:uni_on_annulus}, and \eqref{eq:limitA}  there is a constant $C=C_p>0$ such that for all $\lambda>0$, and $f \in L^p(dw)\cap L^2(dw)$ we have
\begin{equation}\label{eq:K_weak}
    w(\{\mathbf{x} \in \mathbb{R}^N\,:\,|Kf(\mathbf{x})|>\lambda\}) \leq C\lambda^{-1}\|f\|_{L^1(dw)} \ \ \text{ if }p=1,
\end{equation}
\begin{equation}\label{eq:K_strong}
    \|K f\|_{L^p(dw)} \leq C\|f\|_{L^p(dw)} \ \ \text{ if } 1<p<\infty.
\end{equation}
Moreover, if $1<p<\infty$, the operators $K^{\{a\}}$ converge strongly to $K$ in $L^p(dw)$ as $a \to 0$.
\end{theorem}

}
}
{
\subsection{Bessel potential}
For $s>0$ we set
\begin{align*}
    J^{\{s\}}(\mathbf{x})=\mathcal{F}^{-1}((1+\|\cdot\|^2)^{-s\slash 2})(\mathbf{x}).
\end{align*}
By gamma function identity we have
\begin{align*}
    (1+\|\xi\|^2)^{-s\slash 2}=\int_0^{\infty}e^{-t}e^{-t\|\xi\|^2}t^{s\slash 2}\,\frac{dt}{t},
\end{align*}
which leads us to
\begin{equation}\label{eq:subordination}
    J^{\{s\}}(\mathbf{x})=\int_0^{\infty}e^{-t}h_t(\mathbf{x})t^{s\slash 2}\,\frac{dt}{t}
\end{equation}
(see~\eqref{eq:heat_def}). The function $J^{\{s\}}$ is radial, positive and belongs $L^1(dw)$.
As a consequence of~\eqref{eq:subordination} (see i.e.~\cite{Grafakos}), we get the following proposition.
\begin{proposition}\label{propo:bessel_estimate}
Let $M>0$. There is a constant $C=C_{s,M}>0$ such that
\begin{align*}
    0<J^{\{s\}}(\mathbf{x}) \leq C \begin{cases}
    \|\mathbf{x}\|^{s-\mathbf{N}} & \text{ if }\|\mathbf{x}\| \leq 1\slash 2, \quad 0<s<\mathbf N,\\
    -\ln \|\mathbf x\| & \text{ if }\|\mathbf{x}\| \leq 1\slash 2, \quad s=\mathbf N,\\
    1 & \text{ if } \| \mathbf x\|\leq 1/2, \quad s>\mathbf N,\\
    (1+\|\mathbf{x}\|^2)^{-M} & \text{ if }\|\mathbf{x}\|>1/2,  \quad 0<s.
    \end{cases}
\end{align*}
\end{proposition}
Let $J^{\{s\}}(\mathbf x,\mathbf y)=\tau_{\mathbf x}(J^{\{s\}})(-\mathbf y)$. Clearly, by~\eqref{eq:translation-radial}, 
$0<J^{\{s\}}(\mathbf x,\mathbf y)$ for all $\mathbf{x},\mathbf{y} \in \mathbb{R}^N$ and
\begin{align*}
    \int_{\mathbb{R}^N} J^{\{s\}}(\mathbf x,\mathbf y)\, dw(\mathbf x)=\int_{\mathbb{R}^N} J^{\{s\}}(\mathbf x,\mathbf y)\, dw(\mathbf y)=\int_{\mathbb{R}^N} J^{\{s\}}(\mathbf x)\, dw(\mathbf x).
\end{align*}
\begin{lemma}\label{lem:Lipschitz}
Let $0<\delta < s$ and $0<\delta \leq 1$. Then there is a constant $C>0$ such that 
\begin{align*}
    \int_{\mathbb{R}^N} |J^{\{s\}} (\mathbf x,\mathbf y)-J^{\{s\}} (\mathbf x,\mathbf y')|\, dw(\mathbf x)\leq C\min (1, \| \mathbf y-\mathbf y'\|^\delta)  \ \ \text{for all }\mathbf{y},\mathbf{y}' \in \mathbb{R}^N.
\end{align*}
\end{lemma}
\begin{proof}
By~\eqref{eq:subordination} we have
\begin{align}\label{eq:diff_bessel}
    J^{\{s\}}(\mathbf{x},\mathbf{y})-J^{\{s\}}(\mathbf{x},\mathbf{y}')=\int_0^{\infty}e^{-t}\big(h_t(\mathbf{x},\mathbf{y})-h_t(\mathbf{x},\mathbf{y}')\big)t^{s/2}\frac{dt}{t}.
\end{align}
Moreover, by~\eqref{eq:Gauss-Lipschitz} we get
\begin{equation}\label{eq:Holder_heat}
    \int_{\mathbb{R}^N} |h_t(\mathbf x,\mathbf y)-h_t(\mathbf x,\mathbf y')|\,dw(\mathbf x)\leq C\min \Big(1, \big(\| \mathbf y-\mathbf y'\|\slash \sqrt{t}\big)^\delta\Big).
\end{equation}
The lemma is a direct consequence of \eqref{eq:diff_bessel} and \eqref{eq:Holder_heat}.
\end{proof}

\subsection{Auxiliary estimates on \texorpdfstring{$K^{\{a,b\}}$}{Kj}}

The following theorem and proposition were proved in~\cite[Theorem 1.7]{DzH} and~\cite[Proposition 4.4]{DzH} respectively.
\begin{theorem}\label{teo:support}
Let $f \in L^2(dw)$, $\text{\rm supp}\, f \subseteq B(0,r)$, and $\mathbf{x} \in \mathbb{R}^N$. Then  
\begin{equation}\label{eq:inclusion_support}
\supp \tau_{\mathbf{x}}f(-\, \cdot) \subseteq \mathcal{O}(B(\mathbf x,r)),
\end{equation}
where $\mathcal O(B(\mathbf x, r))=\bigcup_{\sigma \in G} B(\sigma (\mathbf x), r))$ is the orbit of the Euclidean closed ball $B(\mathbf x,r)=\{\mathbf y \in \mathbb{R}^N: \| \mathbf x-\mathbf y\|\leq r\}$. 
\end{theorem}
\begin{proposition}\label{propo:compact_supports}
There is a constant $C>0$ such that for any  $r_1,r_2>0$, any $f\in L^1(dw)$ such that ${\rm supp\,} f \subseteq B(0,r_2)$, any continuous radial function $\phi$ such that $\text{\rm supp}\, \phi \subseteq B(0,r_1)$, and for all $\mathbf{y} \in \mathbb{R}^N$ we have
\begin{align*}
\|\tau_{\mathbf{y}}(f * \phi)\|_{L^1(dw)} \leq C (r_1(r_1+r_2))^{\frac{\mathbf{N}}{2}}
\| \phi\|_{L^{\infty}}\|f\|_{L^1(dw)}.
\end{align*}
\end{proposition}

\begin{proposition}\label{propo:j_small}
Let $0 \leq \delta\leq 1$ be such that $\mathbf N\slash 2+\delta<s_0$. Suppose that the function $K$ satisfies~\eqref{eq:K_der2}.
There is a constant $C>0$ such that for all $j \in \mathbb{Z}$ we have
\begin{equation}\label{eq:K_j_bound}
    \int_{\mathbb{R}^N}\sup_{2^{j-1}\leq a\leq b<2^{j+1}}|K^{\{a,b\}}(\mathbf{x},\mathbf{y})|d(\mathbf{x},\mathbf{y})^{\delta}\,dw(\mathbf{y}) \leq C2^{j\delta} \quad \text{for all } \mathbf{x} \in \mathbb{R}^N,
\end{equation}
\begin{equation}\label{eq:K_j_Holder}
    \int_{\mathbb{R}^N}\sup_{2^{j-1}\leq a\leq b<2^{j+1}}|K^{\{a,b\}}(\mathbf{x},\mathbf{y})-K^{\{a,b\}}(\mathbf{x},\mathbf{y}')|\,dw(\mathbf{x}) \leq C2^{-j\delta}\|\mathbf{y}-\mathbf{y}'\|^{\delta} \quad \text{for all } \mathbf{y}, \mathbf y' \in \mathbb{R}^N ,
\end{equation}
\begin{equation}\label{eq:K_j_Holder_22}
    \int_{\mathbb{R}^N}\sup_{2^{j-1}\leq a\leq b<2^{j+1}}|K^{\{a,b\}}(\mathbf{y},\mathbf{x})-K^{\{a,b\}}(\mathbf{y'},\mathbf{x})|\,dw(\mathbf{x}) \leq C2^{-j\delta}\|\mathbf{y}-\mathbf{y}'\|^{\delta} \quad \text{for all } \mathbf{y}, \mathbf y' \in \mathbb{R}^N ,
\end{equation}
\begin{equation}\label{eq:K^n_Holder_22}
    \int_{\mathbb{R}^N}\sup_{2^{j-1}\leq a \leq 2^{j+1}}|K^{\{a\}}(\mathbf{y},\mathbf{x})-K^{\{a\}}(\mathbf{y'},\mathbf{x})|\,dw(\mathbf{x}) \leq C2^{-j\delta}\|\mathbf{y}-\mathbf{y}'\|^{\delta} \quad \text{for all } \mathbf{y}, \mathbf y' \in \mathbb{R}^N .
\end{equation}
\end{proposition}

\begin{proof}
It  suffices to prove the proposition for $j=0$ and then use scaling. Let $\widetilde \phi:[0,\infty)\to \mathbb R$   be defined by relation $\widetilde\phi(\| \mathbf y\|)=\phi (\mathbf y)$.   Set $\Phi(t,\mathbf y)=  -K(\mathbf y)t^{-2} \| \mathbf y\|\widetilde\phi '(t^{-1}\| \mathbf y\|)$. Assume that $2^{-1} \leq a\leq b \leq 2$, then 
\begin{align*}
        K^{\{a,b\}} (\mathbf y)&=K(\mathbf y)\big(\widetilde \phi (b^{-1}\| \mathbf y\|)-\widetilde\phi (a^{-1}\|\mathbf y\|)\big)=\int_a^b K(\mathbf y)\frac{d}{dt} \{\widetilde \phi(t^{-1}\| \mathbf y\|)\}\, dt= \int_a^b \Phi(t,\mathbf y)\, dt,
\end{align*}
 where the integral converges in $L^2(dw)$, because  $t\mapsto \Phi(t,\cdot)$ is a continuous function from $(0,\infty)$ to $L^2(dw)$. Let us denote $\Phi(t,\mathbf x,\mathbf y)=\tau_{\mathbf x}(\Phi (t,\cdot))(-\mathbf y)$. 
 Since the Dunkl translation $\tau_{\mathbf{x}}$ is continuous on $L^2(dw)$, for fixed $\mathbf x\in\mathbb R^N$ we have 
 $$ K^{\{a,b\}}(\mathbf x,\mathbf y)=\int_a^b \Phi(t,\mathbf x,\mathbf y)\, dt,$$
 where the integral converges in $L^2(dw(\mathbf y))$. 
 
 Note that $\text{supp}\, \Phi(t,\cdot)\subseteq \{\mathbf y: t/2\leq \| \mathbf y\| \leq t\}\subseteq B(0,2)$, because $1/2 \leq a\leq t\leq b \leq 2$.  
 Let $\mathbf N/2<s_1\leq s_0$. Set
\begin{align*}
    F(t,\mathbf{y})={\widetilde{\Phi}(t,\cdot)}*(J^{\{s_1\}})(\mathbf{y}), 
\quad \text{where} \quad 
    \widetilde{\Phi}(t,\mathbf{y})=(I-\Delta)^{s_0/2}\Phi(t,\mathbf{y}).
\end{align*}
By the assumption~\eqref{eq:K_der2} (see also~\eqref{eq:K_der2_Dunkl}), $\text{supp}\, \widetilde{\Phi}(t,\cdot)\subseteq B(0,2)$ and $|\widetilde{\Phi} (t,\mathbf x)|\leq C'$ for $1/2\leq t\leq 2$, where the constant $C'$ depends only on the constants $C_\beta$ in \eqref{eq:K_der2} and the (fixed) function $\phi$. 
Consequently, $\widetilde{\Phi}(t,\cdot) \in L^p(dw)$ for $1\leq p\leq \infty$. In particular  there is $C''>0$ independent of $j$ such that
\begin{equation}\label{eq:assumption_usage}
   \sup_{1/2\leq t\leq 2} \|\widetilde{\Phi}(t,\cdot)\|_{L^1(dw)} \leq C'' .
\end{equation}
Let $\psi(\mathbf y)=\phi(\mathbf y)-\phi(2\mathbf y)$.  We write
\begin{align*}
\begin{split}
    J^{\{s_1\}}(\mathbf{y})&=\sum_{\ell \in \mathbb{Z}}\psi(2^{-\ell}\mathbf{y})(J^{\{s_1\}})(\mathbf{y})=:\sum_{\ell \in \mathbb{Z}}\psi^{\{\ell \}}(\mathbf{y}),
\end{split}
\end{align*}
where the convergence is pointwise and in $L^1(dw)$. Recall that $\psi^{\{\ell\}}$ are radial functions. Hence,
\begin{align*}
    F(t,\cdot)=\sum_{\ell \in \mathbb{Z}} \widetilde{\Phi}(t,\cdot)*\psi^{\{\ell\}}
\end{align*}
with convergence in $L^2(dw)$. Therefore for all $\mathbf x\in\mathbb R^N$,
\begin{align*}
    F(t,\mathbf x,\mathbf y):=\tau_{\mathbf x} F(t,-\mathbf y)=\sum_{\ell \in \mathbb{Z}} \tau_{\mathbf x}\big( \widetilde{\Phi}(t,\cdot)*\psi^{\{\ell\}}\big)(-\mathbf y),
\end{align*}
where the series converges in $L^2(dw)$. We shall show that the convergence is in $L^1(dw)$ as well and the $L^1$-norm of the sum is independent of $t\in [1/2,2]$ and $\mathbf x \in \mathbb{R}^N$.

To this end we apply \eqref{eq:assumption_usage} together with  Proposition~\ref{propo:compact_supports}  and Proposition~\ref{propo:bessel_estimate} (with $M>\mathbf{N}$) and obtain
\begin{align*}
     \sum_{\ell \in \mathbb{Z}} \|\tau_{\mathbf{x}}(\widetilde{\Phi}(t,\cdot)*\psi^{\{\ell\}})(\cdot)\|_{L^1(dw)}
      & \leq C \sum_{\ell \in \mathbb{Z}} 2^{\ell  \mathbf{N}/2}(2+2^{\ell})^{\mathbf{N}/2}\|\widetilde{\Phi}(t,\cdot)\|_{L^1(dw)}\|\psi^{\{\ell\}}\|_{L^{\infty}}\\
      &\leq C'\sum_{\ell \leq 0}  2^{\ell \mathbf{N}/2}2^{\ell(s_0-\mathbf{N})} + C'\sum_{\ell >0} 2^{\ell \mathbf N}2^{-\ell M}<\infty.
\end{align*}
Thus,
 \begin{equation}\label{eq:F_0_bound}  \sup_{1/2\leq t\leq 2}\int_{\mathbb{R}^N} |F(t,\mathbf x,\mathbf y)|dw(\mathbf y)\leq C''<\infty, 
\end{equation}
 where $C''$ depends on the constants in \eqref{eq:K_der2}. By the same arguments, 
 \begin{equation}
     \label{eq:F_t_bound}  \sup_{1/2\leq t\leq 2}\int_{\mathbb{R}^N} |F(t,\mathbf x,\mathbf y)|dw(\mathbf x)\leq C''<\infty.
 \end{equation}
 Note that if  $s_1=s_0$ then $F(t,\cdot)=\Phi(t,\cdot)$ and 
 $$ \sup_{1/2\leq a\leq b\leq 2}  | K^{\{a,b\}}(\mathbf x,\mathbf y)|=  \sup_{1/2\leq a\leq b\leq 2}\Big|\int_a^b \Phi(t,\mathbf x,\mathbf y)\, dt\Big|\leq \int_{1/2}^2 |\Phi(t,\mathbf x,\mathbf y)|\, dt. $$
Consequently, applying \eqref{eq:F_0_bound} with $F(t,\cdot)=\Phi(t,\cdot)$, we obtain 
\begin{equation}\label{eq:K^ab}
 \int \sup_{1/2\leq a\leq b\leq 2}  | K^{\{a,b\}}(\mathbf x,\mathbf y)|\, dw(\mathbf y)\leq \int_{\mathbb R^N}\int_{1/2}^2 |\Phi(t,\mathbf x,\mathbf y)|\, dt\, dw(\mathbf y)\leq C'''. 
\end{equation}
Now \eqref{eq:K_j_bound} for $j=0$  follows from \eqref{eq:K^ab}, since, thanks to Theorem~\ref{teo:support}, $K^{\{a,b\}}(\mathbf x,\mathbf y)=0$  if $d(\mathbf x,\mathbf y)>2$.

In order to prove \eqref{eq:K_j_Holder} we take $s_1,s>0$ such that $\mathbf N/2 <s_1$ and $s_0=s_1+s$. Then 
\begin{align*}
    \Phi(t,\cdot)= \widetilde{\Phi}(t,\cdot)* J^{\{s_1\}}*J^{\{s\}}=F(t,\cdot)*J^{\{s\}}.
\end{align*} 
Thus, 
\begin{equation}\label{eq:supab}
\begin{split}
      & \sup_{1/2\leq a\leq b\leq 2} |K^{\{a,b\}}(\mathbf x,\mathbf y)-K^{\{a,b\}}(\mathbf x,\mathbf y')| =  \sup_{1/2\leq a\leq b\leq 2}\Big| \int_a^b \Big(\Phi(t,\mathbf x,\mathbf y)-\Phi(t,\mathbf x,\mathbf y')\Big)\, dt \Big| \\
      &\leq   \sup_{1/2\leq a\leq b\leq 2} \int_{a}^b     \int | F(t,\mathbf x,\mathbf z)| \Big|J^{\{s\}}(\mathbf z,\mathbf y)- J^{\{s\}}(\mathbf z,\mathbf y')\Big|\, dw(\mathbf z)\, dt \\
      &\leq \int_{1/2}^2     \int | F(t,\mathbf x,\mathbf z)| \Big|J^{\{s\}}(\mathbf z,\mathbf y)- J^{\{s\}}(\mathbf z,\mathbf y')\Big|\, dw(\mathbf z)\, dt .\\
\end{split}      
\end{equation}
Integrating~\eqref{eq:supab} with respect to $dw(\mathbf x)$ and using~\eqref{eq:F_t_bound} together with 
 Lemma~\ref{lem:Lipschitz}, we obtain \eqref{eq:K_j_Holder} for $j=0$. The proof of \eqref{eq:K_j_Holder_22} for $j=0$ is identical. 

 In order to prove \eqref{eq:K_j_bound}, \eqref{eq:K_j_Holder}, and \eqref{eq:K_j_Holder_22} for arbitrary $j\in \mathbb Z$ we use scaling. To this end we fix $j\in\mathbb Z$ and 
 write  $G_j(\mathbf x)=2^{j\mathbf N}K(2^j\mathbf x)$. Then $G_j$ satisfies  \eqref{eq:K_der2} with the same constants $C_\beta$. Moreover, one can easily check that 
 $$ K^{\{a,b\}}(\mathbf x)=2^{-j\mathbf N} G_j^{\{2^{-j}a,2^{-j}b\}}(2^{-j}\mathbf x),$$ 
 and, consequently, 
 $$ K^{\{a,b\}}(\mathbf x,\mathbf y)=2^{-j\mathbf N} G_j^{\{2^{-j}a,2^{-j}b\}}(2^{-j}\mathbf x,2^{-j}\mathbf y),$$ 
 Now if $2^{j-1}\leq a\leq b\leq 2^{j+1}$, then $1/2\leq 2^{-j}a\leq 2^{-j}b\leq 2$ and  we apply already proved results to $G_j$ and obtain the desired inequalities.

 We now turn to prove \eqref{eq:K^n_Holder_22}.    If $2^{j-1}\leq a\leq 2^{j+1}$, then we write 
 $$K^{\{a\}}= K^{\{a,2^{j+1}\}}+\sum_{\ell=j+1}^\infty K^{\{2^{\ell},2^{\ell+1}\}},  $$
 where the convergence is in $L^2(dw)$.
  Now application of \eqref{eq:K_j_Holder_22} gives \eqref{eq:K^n_Holder_22}.
\end{proof}

For a cube $Q\subset \mathbb R^N$, let $c_Q$ be its center and  $\diam(Q)$ be the length of its diameter. Let $Q^{*}$ denote the cube with the same center $c_Q$ such that $\diam(Q^{*})=2\diam(Q)$. Let us remind that
\begin{align*}
    \mathcal{O}(Q^{*})=\{\sigma(\mathbf{x})\,:\,x\in Q^{*}, \, \sigma \in G\}.
\end{align*}
The following corollary is a direct consequence of Propositions~\ref{propo:j_small}.
\begin{corollary}\label{coro:j_small_and_large}
There are constants $C, \delta>0$ such that for any $j \in \mathbb{Z}$ and any every cube $Q\subset\mathbb R^N$, and $\mathbf y,\mathbf y'\in Q$ we have
\begin{align*}
    \int_{\mathbb{R}^N \setminus \mathcal{O}(Q^{*})}\sup_{2^{j-1}\leq a\leq b\leq 2^{j+1}}|K^{\{a,b\}}(\mathbf{x},\mathbf{y})&-K^{\{a,b\}}(\mathbf{x},\mathbf{y}')|\,dw(\mathbf{x}) \\
    &\leq C\min(2^{-j\delta}\diam(Q)^{\delta},2^{\delta j}(\diam(Q))^{-\delta}).
\end{align*}
\end{corollary}

\subsection{Proofs of Theorems~\ref{theorem:weak_type_truncated1},~\ref{theorem:strong_type_truncated1}, and~\ref{theorem:weak_type_K}}
\begin{proof}[Proof of Theorem~\ref{theorem:weak_type_truncated1}]
We will prove~\eqref{eq:Psi_nm_weak} first. Take any $0<a\leq b<\infty$. There is a constant $C_1>1$, which depends on the doubling constant in~\eqref{eq:doubling} and $N$, such that
$w(Q)\leq C_1 w(Q')$, where $Q'$ is any sub-cube of $Q$ such that  $\text{\rm diam} (Q')=\text{\rm diam} (Q)\slash 2$.

Let $f\in L^1(dw)\cap L^2(dw)$. Fix $\lambda>0$. We denote by $\mathcal Q_\lambda$ the collection of all maximal (disjoint)  dyadic cubes $Q_\ell$ in $\mathbb R^N$ satisfying
 $$ \lambda< \frac{1}{w(Q_\ell)}\int_{Q_\ell}|f(\mathbf{x})|\,dw(\mathbf{x}).$$
 Then
$$ \frac{1}{w(Q_\ell)} \int_{Q_\ell}|f(\mathbf{x})|\,dw(\mathbf{x})\leq C_1\lambda.$$
   Set $\Omega=\bigcup_{Q_\ell \in\mathcal Q_\lambda} Q_\ell$. Then $w(\Omega)\leq \lambda^{-1}\| f\|_{L^1(dw)}$.  Form the corresponding Calder\'on--Zygmund decomposition of $f$, namely, $f=\mathbf{g}+\mathbf{b}$, where

\begin{equation*}\label{eq:good}
\mathbf{g}(\mathbf{x})=f\chi_{\Omega^c}(\mathbf{x})+\sum_{\ell}w(Q_\ell)^{-1}\Big(\int_{Q_\ell}f(\mathbf{y})\,dw(\mathbf{y})\Big)\chi_{Q_\ell}(\mathbf{x}),
\end{equation*}
\begin{equation*}\label{eq:bad}
\mathbf{b}(\mathbf{x})=\sum_{\ell}\mathbf{b}_\ell(\mathbf{x}), \text{ where }\mathbf{b}_\ell(\mathbf{x})=\left(f(\mathbf{x})-w(Q_\ell)^{-1}\int_{Q_\ell}f(\mathbf{y})\,dw(\mathbf{y})\right)\chi_{Q_\ell}(\mathbf{x}).
\end{equation*}
Clearly, $\mathbf{g}(\mathbf x),\mathbf{b}(\mathbf x)\in L^1(dw(\mathbf x))\cap L^2(dw(\mathbf x))$, $|\mathbf{g}(\mathbf x)|\leq C_1\lambda$, $\| \mathbf{g}\|_{L^2(dw)}^2 \leq C\lambda \| f\|_{L^1}$, $\sum_\ell \| \mathbf{b}_\ell\|_{L^1(dw)}\leq C\| f\|_{L^1(dw)}$.
Further,
\begin{align*}
w(\{\mathbf{x} \in \mathbb{R}^N\,:\,|K^{\{a,b\}}f(\mathbf{x})|> \lambda\}) &\leq w(\{\mathbf{x} \in \mathbb{R}^N\,:\,|K^{\{a,b\}}\mathbf{g}(\mathbf{x})|>\lambda/2\})\\&+w(\{\mathbf{x} \in \mathbb{R}^N\,:\,|K^{\{a,b\}}\mathbf{b}(\mathbf{x})|>\lambda/2\}).
\end{align*}
Since $\|K^{\{a,b\}}\|_{L^2(dw) \to L^2(dw)} \leq C_2$, where $C_2>0$ is independent of $0<a,b<\infty$ (see~Corollary \ref{uniform_L2}), we obtain
\begin{align*}
w(\{\mathbf{x} \in \mathbb{R}^N\,:\,|K^{\{a,b\}}\mathbf{g}(\mathbf{x})|>\lambda/2\}) \leq \frac{4}{\lambda^{2}}C_2^2\|\mathbf{g}\|_{L^2(dw)}^2\leq \frac{C_3}{\lambda}\|f\|_{L^1(dw)}.
\end{align*}
 Define $\Omega^{*}=\mathcal{O}\Big(\bigcup_{Q_\ell\in\mathcal Q_\lambda}Q_\ell^{*}\Big)$. There is a constant $C_2>1$, which depends on the Weyl group, doubling constant, and $N$ such that
\begin{align*}
w(\Omega^{*}) \leq C_2w(\Omega) \leq C_2\lambda^{-1}\|f\|_{L^1(dw)}.
\end{align*}
Thus it suffices to estimate $K^{\{a,b\}} \mathbf{b}(\mathbf x)$ on $\mathbb R^N\setminus \Omega^*$. Let $n_0, n_1 \in \mathbb{Z}$ be such that $2^{n_0-1}<a\leq 2^{n_0}$, $2^{n_1}\leq b<2^{n_1+1}$. 
If $n_0<n_1$, we write 
\begin{equation}\label{K-decomp} K^{\{a,b\}}=K^{\{a,2^{n_0}\}}+\Big(\sum_{j=n_0}^{n_1-1}K^{\{2^{j},2^{j+1}\}}\Big)+K^{\{2^{n_1},b\}},
\end{equation}
otherwise $2^{n_0-1}\leq a\leq b\leq 2^{n_0+1}$ and we consider just the single kernel $K^{\{a,b\}}$. 

Since $\sum_{\ell}\mathbf{b}_\ell$ converges to $\mathbf{b}$ in $L^2(dw)$, $K^{\{a,b\}}\mathbf{b}=\sum_{\ell} K^{\{a,b\}}\mathbf{b}_\ell$ with convergence in $L^2(dw)$. 
So we have
\begin{align*}
|K^{\{a,b\}}\mathbf{b}(\mathbf{x})| \leq \sum_{\ell}\Big( |K^{\{a,2^{n_0}\}}\mathbf{b}_\ell(\mathbf x)|+\Big(\sum_{j=n_0}^{n_1-1}|K^{\{2^{j},2^{j+1}\}}\mathbf{b}_\ell(\mathbf x)|\Big)+|K^{\{2^{n_1},b\}}\mathbf{b}_\ell(\mathbf x)|\Big)  .
\end{align*}
By the fact that $\supp \mathbf{b}_\ell \subseteq Q_\ell$ and $\int_{\mathbb{R}^N} \mathbf{b}_\ell(\mathbf{y})\,dw(\mathbf{y})=0$,  we get
\begin{equation}\label{eq:int_ell_j}\begin{split}
&\int_{\mathbb{R}^N \setminus \Omega^{*}}|K^{\{2^{j},2^{j+1}\}}\mathbf{b}_\ell(\mathbf{x})|\,dw(\mathbf{x})
 =\int_{\mathbb{R}^N \setminus \Omega^{*}}\left|\int_{Q_\ell} K^{\{2^{j}, 2^{j+1}\}}(\mathbf{x},\mathbf{y})\mathbf{b}_{\ell}(\mathbf{y})\,dw(\mathbf{y})\right|\,dw(\mathbf{x})\\
&\leq \int_{\mathbb{R}^N \setminus \mathcal O(Q_\ell^{*})}\left|\int_{Q_\ell}
 \Big(K^{\{2^j, 2^{j+1}\}}(\mathbf{x},\mathbf{y})-K^{\{2^j,2^{j+1}\}} (\mathbf x,c_{Q_\ell})\Big)\mathbf{b}_{\ell}(\mathbf{y})\,dw(\mathbf{y})\right|\,dw(\mathbf{x})\\
& \leq C \min \Big(2^{-j\delta}\diam(Q_\ell)^{\delta},2^{\delta j}\diam(Q_{\ell})^{-\delta}\Big)\| \mathbf{b}_\ell\|_{L^1(dw)},
\end{split}\end{equation}
where  in the last inequality we have used Corollary~\ref{coro:j_small_and_large} with $\delta>0$ small enough.
Similarly, 
\begin{equation}
 \int_{\mathbb{R}^N\setminus \Omega^{*}}|K^{\{a,2^{n_0}\}}\mathbf{b}_\ell(\mathbf{x})|\,dw(\mathbf{x})\leq C \min \Big(2^{-n_0\delta}\diam(Q_\ell)^{\delta},2^{n_0\delta }\diam(Q_{\ell})^{-\delta}\Big)\| \mathbf{b}_\ell\|_{L^1(dw)},
\end{equation}
\begin{equation}\label{eq:end}
 \int_{\mathbb{R}^N\setminus \Omega^{*}}|K^{\{2^{n_1},b\}}\mathbf{b}_\ell(\mathbf{x})|\,dw(\mathbf{x})\leq C \min \Big(2^{-n_1\delta}\diam(Q_\ell)^{\delta},2^{n_1\delta }\diam(Q_{\ell})^{-\delta}\Big)\| \mathbf{b}_\ell\|_{L^1(dw)}.
\end{equation}
Summing  up the inequalities~\eqref{eq:int_ell_j}--\eqref{eq:end} over $\ell \in \mathbb{Z}$ we end up with
$$\int_{\mathbb R^N\setminus \Omega^*} | K^{\{a,b\}}\mathbf{b}(\mathbf x)|\,dw(\mathbf x)\leq C \sum_{\ell} \| \mathbf{b}_\ell\|_{L^1}\leq C\| f\|_{L^1(dw)},  $$
with $C$ independent of $0<a\leq b<\infty$. Consequently, by the Chebyshev inequality, this completes the proof of weak type $(1,1)$ of the operator $K^{\{a,b\}}$.

In order to prove~\eqref{eq:Psi_n_weak}, let us note that for any $f\in L^2(dw)\cap L^1(dw)$ and any $a>0$ there is an increasing  sequence $m_j\to\infty$, $m_j>a$,  such that  $K^{\{a\}} f=\lim_{j\to\infty}  K^{\{a,m_j\}}f $ with convergence in $L^2(dw)$ and almost everywhere. Therefore, up to a set of measure zero we have 
\begin{equation}\label{eq:cup_cap_cup}
    \{\mathbf{x} \in \mathbb{R}^N\,:\,|K^{\{a\}}f(\mathbf{x})|>\lambda\}= 
    \bigcup_{n=1}^\infty \bigcap_{j \geq n}\{\mathbf{x} \in \mathbb{R}^N\,:\,|K^{\{a,m_j\}}f(\mathbf{x})|>\lambda\}.
\end{equation}
 So, thanks to~\eqref{eq:cup_cap_cup} and~\eqref{eq:Psi_nm_weak}, we have
\begin{align*}
    w(\{\mathbf{x} \in \mathbb{R}^N\,:\,|K^{\{a\}}f(\mathbf{x})|>\lambda\}) &\leq \sup_{a \leq b}w(\{\mathbf{x} \in \mathbb{R}^N\,:\,|K^{\{a,b\}}f(\mathbf{x})|>\lambda\}) \leq \frac{C}{\lambda}\|f\|_{L^1(dw)}.
\end{align*}
\end{proof}

\begin{proof}[Proof of Theorem~\ref{theorem:strong_type_truncated1}]
First we note that  by \eqref{K-decomp} and \eqref{eq:K_j_bound} with $\delta=0$ we have 
$$\int |K^{\{a,b\}}(\mathbf x,\mathbf y)|\, dw(\mathbf x)\leq C_{a,b}, \quad  \int |K^{\{a,b\}}(\mathbf x,\mathbf y)|\, dw(\mathbf y)\leq C_{a,b}.$$
Thus the operators $K^{\{a,b\}}$ are bounded on $L^p(dw)$. Moreover, $(K^{\{a,b\}})^*=(K^*)^{\{a,b\}}$, where $K^*(\mathbf x)=\overline{K(-\mathbf x)}$. Now for fixed $1<p<\infty$ the uniform bound of the operators $K^{\{a,b\}}$ on $L^p(dw)$ 
follows from  interpolation, Corollary \ref{uniform_L2}, Theorem~\ref{theorem:weak_type_truncated1}, and duality. 

Furthermore,~\eqref{eq:Psi_n_strong} and the strong convergence of $K^{\{a,b\}}$ to $K^{\{a\}}$ on the space  $L^p(dw)$ for $1<p<2$ as ($b\to\infty$) also follows from the Marcinkiewicz interpolation theorem and Corollary \ref{eq:L2strong}. In order to prove~\eqref{eq:Psi_n_strong} for $p>2$, let us show first that for $f \in \mathcal{S}(\mathbb{R}^N)$ the function  $(a,\infty)\ni b\mapsto K^{\{a,b\}}f\in L^p(dw)$ satisfies the  Cauchy condition as $b\to\infty$. Clearly,
\begin{align*}
    &\|K^{\{a,b_1\}}f-K^{\{a,b_2\}}f\|_{L^p(dw)} \\&\leq \|K^{\{a,b_1\}}f-K^{\{a,b_2\}}f\|_{L^2(dw)}^{1/(p-1)}\|K^{\{a,b_1\}}f-K^{\{a,b_2\}}f\|_{L^{2p}(dw)}^{1-1/(p-1)}.
\end{align*}
Thanks to \eqref{eq:L2strong} we have
\begin{align*}
    \lim_{b_1,b_2 \to \infty}\|K^{\{a,b_1\}}f-K^{\{a,b_2\}}f\|_{L^2(dw)}^{1/(p-1)}=0,
\end{align*}
moreover, by ~\eqref{eq:Psi_nm_strong},  we get 
\begin{align*}
    \|K^{\{a,b_1\}}f-K^{\{a,b_2\}}f\|_{L^{2p}(dw)}^{1-1/(p-1)} \leq C\|f\|_{L^{2p}(dw)}^{1-1/(p-1)}<\infty.
\end{align*}
Consequently,
\begin{equation}\label{eq:convergence_on_schwartz}
    \lim_{b_1,b_2 \to \infty}\|K^{\{a,b_1\}}f-K^{\{a,b_2\}}f\|_{L^p(dw)} =0
\end{equation}
for $f \in \mathcal{S}(\mathbb{R}^N)$. In order to prove~\eqref{eq:convergence_on_schwartz} for $f \in L^p(dw)$, it is enough to take a sequence $\mathcal{S}(\mathbb{R}^N) \ni f_{\ell} \to f$ in $L^p(dw)$ and  write
\begin{align*}
    \|K^{\{a,b_1\}}f-K^{\{a,b_2\}}f\|_{L^p(dw)} &\leq \|K^{\{a,b_1\}}(f-f_\ell)\|_{L^p(dw)}+\|K^{\{a,b_1\}}f_\ell-K^{\{a,b_2\}}f_\ell\|_{L^p(dw)}\\&+\|K^{\{a,b_2\}}(f_{\ell}-f)\|_{L^p(dw)},
\end{align*}
then use~\eqref{eq:Psi_nm_strong} for the first and third summand  and~\eqref{eq:convergence_on_schwartz} for the second one.
\end{proof}

\begin{proof}[Proof of Theorem~\ref{theorem:weak_type_K}]
In order to prove~\eqref{eq:K_weak}, we use the same argument as in the second part of the proof of Theorem~\ref{theorem:weak_type_truncated1} (see~\eqref{eq:cup_cap_cup}). In order to prove~\eqref{eq:K_strong}, see the proof of Theorem~\ref{theorem:strong_type_truncated1}.
\end{proof}

\section{Maximal function associated with singular integral}

Let 
\begin{equation}\label{eq:maximal}
    K^{*}f(\mathbf{x})=\sup_{a>0}|K^{\{a\}}f(\mathbf{x})|.
\end{equation}
The goal of this section is to prove the following theorem.

\begin{theorem}\label{theorem:maximal}
Under~\eqref{eq:K_der2},~\eqref{eq:uni_on_annulus}, and~\eqref{eq:limitA}, the operator $K^{*}$ is of weak type $(1,1)$ and it is bounded on $L^p(dw)$ for $1<p<\infty$.
\end{theorem}

{ 
\subsection{Cotlar type inequality}

Let
\begin{align*}
\mathcal{M}_{\rm HL}f(\mathbf{x})=\sup_{\mathbf{x} \in B}\frac{1}{w(B)}\int_{B}|f(\mathbf{y})|\,dw(\mathbf{y}),
\end{align*}
where the supremum is taken over all Euclidean balls $B$ which contain $\mathbf{x}$, be the non-centered  Hardy-Littlewood maximal function defined  on the space of homogeneous type $(\mathbb R^N, \|\mathbf x-\mathbf y\|, dw)$. The following lemma is in the spirit of Cotlar's inequality (cf.~\cite[Lemma 5.15]{Duo}).

\begin{lemma}
 Let $p \in [1,\infty)$. There is a constant $C>0$ such that for all $f \in L^p(dw) \cap L^{\infty}$ and $\mathbf{x} \in \mathbb{R}^N$ we have
\begin{equation}\label{eq:Cotlar}
    K^*f(\mathbf{x}) \leq C\Big(\sum_{\sigma \in G}(\mathcal{M}_{\rm HL}Kf)(\sigma(\mathbf{x}))+\|f\|_{L^{\infty}}\Big).
\end{equation}
\end{lemma}
\begin{proof}
We assume additionally that $f\in L^2(dw)$. Then this assumption can be easily relaxed  by a density argument. Let $\varphi \in C_c^{\infty}(\mathbb{R}^N)$ be a radial function such that $\supp \varphi \subseteq B(0,1)$ and $\int_{\mathbb{R}^N} \varphi\, dw=1$. Fix $a>0$ and set $\widetilde{K^{\{a\}}}=K-K^{\{a\}}$. Let as remind that $\varphi_a(\mathbf x)=a^{-\mathbf N} \varphi(\mathbf x /a)$. Then 
\begin{equation*}
    \begin{split}
        K^{\{a\}}f&= (K^{\{a\}}-\varphi_a*\varphi_a*K^{\{a\}})*f+\varphi_a*\varphi_a*(Kf)-((\varphi_a*(\widetilde{K^{\{a\}}}\varphi_a))*f\\&=:J_1+J_2-J_3.
    \end{split}
\end{equation*}
Clearly, by Lemma~\ref{lem:nonradial_estimation},  $|J_2|\leq C \sum_{\sigma\in G} ( \mathcal M_{\rm HL}Kf)(\sigma \cdot)$. In order to estimate $J_3$ we note that $\text{supp}\, \widetilde{K^{\{a\}}}\varphi_a \subset B(0,2a)$, so by Cauchy--Schwarz inequality and Theorem~\ref{theorem:strong_type_truncated1} together with Lemma~\ref{lem:nonradial_estimation} we have 
\begin{equation}\label{eq:CS_app}
    \|\widetilde{K^{\{a\}}}\varphi_a \|_{L^1(dw)}\leq C w(B(0,2a))^{1\slash 2} \| \varphi_a\|_{L^2(dw)}\leq C.
\end{equation}
Applying  Proposition \ref{propo:compact_supports} we get 
$$ \int_{\mathbb{R}^N} |(\varphi_a*(\widetilde{K^{\{a\}}}\varphi_a))(\mathbf x,\mathbf y)|\, dw(\mathbf y)\leq Ca^{\mathbf N} \|\varphi_a\|_{L^\infty} \cdot \|\widetilde{K^{\{a\}}}\varphi_a \|_{L^1(dw)}\leq C, $$
which, together with~\eqref{eq:CS_app}, implies $|J_3|\leq C\| f\|_{L^\infty}.$ Finally, in order to evaluate $J_1$, we consider the integral kernel of the operator  $(K^{\{a\}}-\varphi_a*\varphi_a*K^{\{a\}})$. Using~\eqref{eq:K^n_Holder_22} and Lemma~\ref{lem:nonradial_estimation} we get 
\begin{align*}
        \int_{\mathbb{R}^N}| & K^{\{a\}}(\mathbf x,\mathbf y) -(\varphi_a*\varphi_a*K^{\{a\}})(\mathbf x,\mathbf y)|\, dw(\mathbf y)\\
        &= \int_{\mathbb{R}^N} \Big|\int_{\mathbb{R}^N} (\varphi_a*\varphi_a)(\mathbf x,\mathbf z) (K^{\{a\}}(\mathbf x,\mathbf y)-K^{\{a\}}(\mathbf z,\mathbf y))\, dw(\mathbf z)\Big|\, dw(\mathbf y)\\
        &\leq \int_{\mathbb{R}^N}\int_{\mathbb{R}^N} |(\varphi_a*\varphi_a)(\mathbf x,\mathbf z)||K^{\{a\}} (\mathbf x,\mathbf y)-K^{\{a\}}(\mathbf z,\mathbf y)|\, dw(\mathbf y)\, dw(\mathbf z)\\
        &\leq \int_{\mathbb{R}^N} |(\varphi_a*\varphi_a)(\mathbf x,\mathbf z)|  a^{-\delta} \| \mathbf x-\mathbf z\|^{\delta} \, dw(\mathbf z) \\
        &\leq \int_{\mathcal{O}(B(\mathbf{x},4a))} w(B(\mathbf x,2a))^{-1} \Big(1+\frac{\| \mathbf x-\mathbf z\|}{a}\Big)^{-2} a^{-\delta} \| \mathbf x-\mathbf z\|^{\delta } \, dw(\mathbf z) \leq C,  
\end{align*}
which gives $|J_1|\leq C\| f\|_{L^\infty}$. 
\end{proof}
}

\subsection{Proof of Theorem~\ref{theorem:maximal}}

\begin{proof}
It is enough to prove that $ K^{*}$ is of weak type $(p,p)$ for $1 \leq p <\infty$. Let $f\in L^p(dw)\cap L^2(dw)$. Fix $\lambda>0$. We denote by $\mathcal Q_\lambda$ the collection of all maximal (disjoint)  dyadic cubes $Q_\ell$ in $\mathbb R^N$ satisfying
\begin{equation}\label{eq:lambda_upper}
    \lambda^p< \frac{1}{w(Q_\ell)}\int_{Q_\ell}|f(\mathbf{x})|^p\,dw(\mathbf{x}).
\end{equation}
 Then
 \begin{equation}\label{eq:lambda_lower}
\frac{1}{w(Q_\ell)} \int_{Q_\ell}|f(\mathbf{x})|^p\,dw(\mathbf{x})\leq C_1\lambda^p.
 \end{equation}
Set $\Omega=\bigcup_{Q_\ell \in\mathcal Q_\lambda} Q_\ell$. Thanks to~\eqref{eq:lambda_upper} we have
   \begin{equation}\label{eq:sum_Q_j}
\sum_{\ell}w(Q_\ell)=w(\Omega)\leq \lambda^{-p}\| f\|_{L^p(dw)}^p.
   \end{equation}
Form the corresponding Calder\'on--Zygmund decomposition of $f$, namely, $f=\mathbf{g}+\mathbf{b}$, where

\begin{equation*}\label{eq:good_max}
\mathbf{g}(\mathbf{x})=f\chi_{\Omega^c}(\mathbf{x})+\sum_{\ell}w(Q_\ell)^{-1}\Big(\int_{Q_\ell}f(\mathbf{y})\,dw(\mathbf{y})\Big)\chi_{Q_\ell}(\mathbf{x}),
\end{equation*}
\begin{equation*}\label{eq:bad_max}
\mathbf{b}(\mathbf{x})=\sum_{\ell}\mathbf{b}_\ell(\mathbf{x}), \text{ where }\mathbf{b}_\ell(\mathbf{x})=\left(f(\mathbf{x})-w(Q_\ell)^{-1}\int_{Q_\ell}f(\mathbf{y})\,dw(\mathbf{y})\right)\chi_{Q_\ell}(\mathbf{x}).
\end{equation*}
Clearly, $\|\mathbf{g}\|_{L^p(dw)}+\|\mathbf{b}\|_{L^p(dw)} \leq C\| f\|_{L^p(dw)}$, $\|\mathbf{g}\|_{L^2(dw)}+\|\mathbf{b}\|_{L^2(dw)} \leq C\| f\|_{L^2(dw)}$, and  $|\mathbf{g}(\mathbf x)|\leq C_1^{ 1/p}\lambda$. Further,
\begin{align*}
w(\{\mathbf{x} \in \mathbb{R}^N\,:\,|K^{*}f(\mathbf{x})|> \lambda) &\leq w(\{\mathbf{x} \in \mathbb{R}^N\,:\,|K^{*}\mathbf{b}(\mathbf{x})|>\lambda/2\})\\&+w(\{\mathbf{x} \in \mathbb{R}^N\,:\,|K^{*}\mathbf{g}(\mathbf{x})|>\lambda/2\})=:S_1+S_2.
\end{align*}

Let us estimate $S_1$ first.  Define $\Omega^{*}=\mathcal{O}\Big(\bigcup_{Q_\ell\in\mathcal Q_\lambda}Q_\ell^{*}\Big)$. There is a constant $C_2>1$, which depends on the Weyl group, doubling constant, and $N$ such that
\begin{align}\label{eq:OmegaStar}
w(\Omega^{*}) \leq C_2w(\Omega) \leq C_2\lambda^{-p}\|f\|^p_{L^p(dw)}.
\end{align}
Thus it suffices to estimate $K^{*}\mathbf{b}(\mathbf x)$ on $\mathbb R^N\setminus \Omega^*$. Note that $\sum_{\ell}\mathbf{b}_\ell$ converges to $\mathbf{b}$ in $L^2(dw)$.  Let us remind that $c_{Q_\ell}$ is the center of $Q_j$. We write
\begin{equation}\label{eq:on_bad_computation}
    \begin{split}
        &\|K^{*}\mathbf{b}\|_{L^1(\mathbb{R}^N \setminus \Omega^{*},dw)}=\int_{\mathbb{R}^N \setminus \Omega^{*}}\sup_{a>0}\Big|\sum_{\ell}\int_{Q_{\ell}}K^{\{a\}}(\mathbf{x},\mathbf{y})\mathbf{b}_{\ell}(\mathbf{y})\,dw(\mathbf{y})\Big|\,dw(\mathbf{x}) \\&\leq \int_{\mathbb{R}^N \setminus \Omega^{*}}\sup_{a>0}\Big|\sum_{\ell}\int_{Q_{\ell}}(K^{\{a\}}(\mathbf{x},\mathbf{y})-K^{\{a\}}(\mathbf{x},c_{Q_{\ell}}))\mathbf{b}_{\ell}(\mathbf{y})\,dw(\mathbf{y})\Big|\,dw(\mathbf{x}) \\& \leq \sum_{\ell}\sum_{j \in \mathbb{Z}}\int_{Q_{\ell}}|\mathbf{b}_{\ell}(\mathbf{y})|\int_{\mathbb{R}^N \setminus {\mathcal O(Q_{\ell}^{*})}}\sup_{2^{j-1} \leq a' \leq b' \leq 2^{j+1}}|K^{\{a',b'\}}(\mathbf{x},\mathbf{y})-K^{\{a',b'\}}(\mathbf{x},c_{Q_{\ell}})|\,dw(\mathbf{x})\,dw(\mathbf{y})\\
& \leq C \sum_{\ell}\sum_{j \in \mathbb{Z}}\min \Big(2^{-j\delta}\diam(Q_\ell)^{\delta},2^{\delta j}\diam(Q_{\ell})^{-\delta}\Big)\| \mathbf{b}_\ell\|_{L^1(dw)},
\end{split}\end{equation}
where  in the last inequality we have used Corollary~\ref{coro:j_small_and_large} with $\delta>0$ small enough. By H\"older's inequality and~\eqref{eq:lambda_lower} we have
\begin{equation}\label{eq:sum_bad}
    \sum_{\ell}\|\mathbf{b}_{\ell}\|_{L^1(dw)} \leq \sum_{\ell}w(Q_j)\left(w(Q_{\ell})^{-1}\int_{Q_\ell}|\mathbf{b}_{\ell}(\mathbf{x})|^p\,dw(\mathbf{x})\right)^{1/p} \leq C\lambda\sum_{\ell}w(Q_{\ell}).
\end{equation}
Combining~\eqref{eq:on_bad_computation},~\eqref{eq:sum_bad}, and~\eqref{eq:sum_Q_j} we get
\begin{equation}\label{eq:K_star_L_1}
    \|K^{*}\mathbf{b}\|_{L^1(\mathbb{R}^N \setminus \Omega^{*},\,dw)} \leq C \lambda^{-p+1}\|f\|_{L^p}^p.
\end{equation}
Chebyshev's inequality applied to  ~\eqref{eq:K_star_L_1} together with \eqref{eq:OmegaStar}  imply $S_1\leq C \lambda^{-p}\| f\|_{L^p(dw)}^p$. 

In order to estimate $S_2$, let us note that it is enough to obtain a proper bound for
\begin{equation*}
  w(\{\mathbf{x} \in \mathbb{R}^N\,:\,|K^{*}\mathbf{g}(\mathbf{x})|>2C_1^{1/p}C_3\lambda\}),
\end{equation*}
where $C_3$ is the constant from~\eqref{eq:Cotlar}. Thanks to the fact that $\|\mathbf{g}\|_{L^{\infty}} \leq C_1^{1/p}\lambda$ and~\eqref{eq:Cotlar} we have
\begin{equation*}
  w(\{\mathbf{x} \in \mathbb{R}^N\,:\,\big|K^{*}\mathbf{g}(\mathbf{x})\big|>2C_1^{1/p}C_3\lambda\}) \leq w(\{\mathbf{x} \in \mathbb{R}^N\,:\,|\sum_{\sigma \in G}\mathcal{M}_{\rm HL}K\mathbf{g}(\sigma(\mathbf{x}))|>C_1^{1/p}\lambda\}).
\end{equation*}
If $p>1$, then  let us remind that $\mathcal{M}_{\rm HL}$ and $K$ are bounded operators on $L^p(dw)$ (see Theorem~\ref{theorem:weak_type_K}), so $S_2\leq C  \lambda^{-p}\| f\|_{L^p(dw)}^p$. 
If $p=1$, then $\|\mathbf{g}\|_{L^2(dw)} \leq C\lambda \| f\|_{L^1(dw)}$ and, consequently, 
\begin{align*}
   w(\{\mathbf{x} \in \mathbb{R}^N\,:\,|\sum_{\sigma \in G}\mathcal{M}_{\rm HL}K\mathbf{g}(\sigma(\mathbf{x}))|>C_1^{1/p}\lambda\}) \leq C\lambda^{-2}\|\mathbf{g}\|_{L^{2}(dw)}^{2} \leq C\lambda^{-1}\|f\|_{L^1(dw)}.
\end{align*}
\end{proof}
}


\begin{thebibliography}{99}

\bibitem{ADzH} 
J.-Ph. Anker, J. Dziuba\'nski, A. Hejna,
	\emph{Harmonic functions, conjugate harmonic functions and the Hardy space $H^1$ in the rational Dunkl setting}, J. Fourier Anal. Appl. 25 (2019), 2356--2418.

\bibitem{AH}
B. Amri, A. Hammi,
\emph{Dunkl-Schr\"odinger operators\/},
{Complex Anal. Oper. Theory (2018).}

\bibitem{deJeu}
	M.F.E. de Jeu,
	\emph{The Dunkl transform\/},
	Invent. Math. 113 (1993), 147--162.
	
\bibitem{RoeslerDeJeu}
	M. de Jeu, M. R\"osler,
	\emph{Asymptotic analysis for the Dunkl kernel},
	J. Approx. Theory 119 (2002), no. 1,
	110--126.

\bibitem{Dunkl0}
	C.F.~Dunkl,
	\emph{Reflection groups and orthogonal polynomials on the sphere},
 	Math. Z. 197 (1988), no. 1,
 	33--60.

\bibitem{Dunkl}
	C.F. Dunkl,
	\emph{Differential-difference operators associated to reflection groups\/},
	Trans. {Amer}. Math. 311 {(1989), no. 1,} 167--183{.}
	
\bibitem{Dunkl3}
	C.F.~Dunkl,
	\emph{Hankel transforms associated to finite reflection groups},
in: {\it Proc. of the special session on hypergeometric functions on domains of positivity, Jack polynomials and applications,}
	 Proceedings, Tampa 1991, Contemp. Math. 138 (1989),
	 123--138.
     
\bibitem{Dunkl2}
	C.F.~Dunkl,
	\emph{Integral kernels with reflection group invariance},
	Canad. J. Math. 43 (1991), no. 6,
	1213--1227.

\bibitem{Duo}
    J. Duoandikoetxea,
    \emph{Fourier Analysis},
    Graduate Studies in Mathematics, 29. American Mathematical Society, Providence, RI, 2001. xviii+222 pp. ISBN: 0-8218-2172-542-01.
    
    
\bibitem{DzH1}
	J. Dziuba\'nski, A. Hejna,
	\emph{Remark on atomic decompositions for Hardy space $H^1$ in the rational Dunkl setting},
	\href{https://doi.org/10.4064/sm180618-25-11}{[doi:10.4064/sm180618-25-11]}, to appear in Studia Math.
	

\bibitem{DzH} 
    J. Dziuba\'nski and A. Hejna, 
    \emph{H\"ormander's multiplier theorem for the Dunkl transform},
     Journal of Functional Analysis 277 (2019), 2133-2159.
     
     
\bibitem{GR}
	L. Gallardo, C.  Rejeb,
	\emph{A new mean value property for harmonic functions
relative to the Dunkl-Laplacian operator and applications\/},
	Trans. Amer. Math. Soc. 368 (2015), {no. 5,} 3727--3753.
	
	
\bibitem{Grafakos}
	L. Grafakos,
	\emph{Modern Fourier Analysis},
	3rd edition, Graduate Texts in Mathematics, 250. Springer, New York,
	2014.
    
\bibitem{Roesler2}
	M. R\"osler,
	\emph{Generalized Hermite polynomials and the heat equation for Dunkl operators\/},
	Comm. Math. Phys. {192} (1998), 519--542.

\bibitem{Roesle99}
	M. R\"osler,
	\emph{Positivity of Dunkl's intertwining operator\/},
	Duke Math. J. 98 (1999), no. 3, 445--463.

\bibitem{Roesler2003}
	M. R\"osler,
	\emph{A positive radial product formula for the Dunkl kernel\/},
	Trans. Amer.Math. Soc. 355 (2003), no. 6, 2413--2438{.}
	
\bibitem{Roesler3}
	M. R\"osler:
	\emph{Dunkl operators (theory and applications).
	In: Koelink, E., Van Assche, W. (eds.)
	Orthogonal polynomials and special functions} (Leuven, 2002), 93--135.
	Lect. Notes Math. 1817, Springer-Verlag (2003).
	
\bibitem{Roesler-Voit}
	M.~R\"osler, M. Voit,
	\emph{Dunkl theory, convolution algebras, and related Markov processes\/},
in \textit{Harmonic and stochastic analysis of Dunkl processes\/},
{P. Graczyk, M. R\"osler, M. Yor (eds.), 1--112, Travaux en cours 71,}
Hermann, Paris, 2008.

\bibitem{St1}
E.M. Stein,
\emph{Singular integral and differentiability properties of functions\/},
Princeton {Math.} Series 30, Princeton Univ. Press, 1970.

\bibitem{St2}
E.M. Stein,
\emph{Harmonic analysis\/}
{(\textit{real variable methods, orthogonality and oscillatory integrals\/})}, {Princeton Math. Series 43}, Princeton Univ. Press, 1993.
	
\bibitem{ThangaveluXu}
	S. Thangavelu, Y. Xu,
	\emph{Convolution operator and maximal function for the Dunkl transform},
	 J. Anal. Math. 97 (2005),
	 25--55.
	 
\bibitem{Trimeche2002}
	K. Trim\'eche,
    \emph{Paley-Wiener theorems for the Dunkl transform and Dunkl translation operators},
     Integral Transforms Spec. Funct. 13 (2002), no. 1, 
     17--38.
\end{thebibliography}
\end{document}